\documentclass{amsart}

\usepackage{amsmath,amsthm,amssymb,amsfonts,parskip, mathtools, yfonts}
\usepackage{tipa}

\usepackage[nobysame]{amsrefs}
\usepackage{color}
\usepackage{enumitem}
\usepackage[all]{xy}

\usepackage{endnotes}
\usepackage[mathscr]{eucal}

\def\N{{\mathbb N}}

\theoremstyle{plain}
\newtheorem*{Theorem}{Theorem}
\newtheorem{theorem}{Theorem}[section]
\newtheorem{proposition}[theorem]{Proposition}
\newtheorem{lemma}[theorem]{Lemma}
\newtheorem*{Question}{Question}
\newtheorem{corollary}[theorem]{Corollary}

\theoremstyle{definition}
\newtheorem{definition}[theorem]{Definition}
\newtheorem{example}[theorem]{Example}
\newtheorem{conjecture}[theorem]{Conjecture}

\theoremstyle{remark}
\newtheorem{remark}[theorem]{Remark}

\numberwithin{equation}{theorem}

\def\Fa{ \,\,\forall \,\, }

\def\lra{\longrightarrow}

\def\Supp{\mbox{ Supp}}

\def\surj{\twoheadrightarrow}
\def\inj{\hookrightarrow}

\DeclarePairedDelimiter{\brackets}{(}{)}
\newcommand{\T}[2]{\operatorname{{\tau}}_{{#2}} (#1)}
\newcommand{\End}[2]{\operatorname{End}_{#1} (#2)}
\newcommand{\MO}{M \otimes_R M^*} 

\renewcommand{\hom}[3]{\operatorname{Hom}_{#1} (#2, #3)}

\newcommand{\Z}{\operatorname{{Z}} \brackets}
\newcommand{\IM}{\operatorname{Im} \brackets}

\newcommand{\Ass}{\operatorname{Ass} \brackets}
\newcommand{\Ann}{\operatorname{Ann}_R \brackets}

\newcommand{\Ext}[4]{\operatorname{Ext}^{#1}_{#2} (#3, #4)}
\newcommand{\grade}{\operatorname{grade} \brackets}

\begin{document}

\title[Trace ideals and Centers]{Trace Ideals  and Centers of Endomorphism Rings \\ 
of   Modules over Commutative Rings }

\author{Haydee Lindo}
\curraddr{Dept. of Mathematics \& Statistics, Williams College, Williamstown, MA, USA}
\email{08hml@williams.edu}

\thanks{This research was partially supported by NSF grant DMS-1503044}

\date{\today}



\maketitle

\begin{abstract} Let $R$ be a commutative Noetherian ring and $M$ a finitely generated $R$-module. Under various hypotheses, it is proved that the center of $\End R M$ coincides with the endomorphism ring of the trace ideal of $M$. These results are exploited to establish results for balanced  and rigid modules, and to settle certain cases of a conjecture of Huneke and Wiegand. 
\end{abstract}

\section{Introduction} 
Let $R$ be a commutative ring and $M$ a finitely generated $R$-module. The trace ideal of $M$, denoted $\T R M$,  is the ideal  $\sum \alpha (M)$ as $\alpha $ ranges over $M^*:= \hom R M R $. 
We are interested in the connection between the properties of $M$ and those of $\T R M$. For special classes of modules the corresponding trace ideals demonstrate remarkable properties. For example, a finitely generated ideal is the trace ideal of a projective module if and only if it is idempotent, and  $\T R M = R$ if and only if every finitely generated $R$-module is a homomorphic image of a direct sum of copies of $M$; see \cite{Lam}, \cite{Whitehead1980} and Section \ref{properties}.

This work was motivated by hints in the literature about the relationship between $\T R M$ and the center of $\End R M$. In \cite{MaxOrders}, Auslander and Goldman   show that when $\T R M  = R$, each endomorphism in the center of $\End R M$ is given by multiplication by a unique ring element (see also \cite[Exercise 95]{Kaplansky1954}). That is, $\Z{\End R M}  = R.$

Two of our central results clarify this connection, first in the case when $M$  is reflexive and faithful and second in the case where  $\T R M$ has positive grade; see Theorems \ref{main} and \ref{main2}. In both cases, $\End R {\T R M}$ equals the center of the endomorphism ring of an associated module (when the rings are viewed as subrings of the total ring of quotients); see Corollaries \ref{inQ} and \ref{identify}.

\begin{Theorem}
Let $R$ be a Noetherian ring and $M$ a finitely generated $R$-module.
\begin{enumerate}[label=(\roman*)]
\item If $M$ is reflexive and faithful,  then there is a canonical isomorphism \\of $R$-algebras
\[{\End R {\T R M}}  \cong \Z {\End R M}\]
\item If $\T R M$ has positive grade, then there are canonical  isomorphisms of $R$-algebras
\[ \End R {\T R M} \cong \End R {\T R {M^*}} \cong \Z{\End R {M^*}} .\]

\end{enumerate} 
\end{Theorem}

These and other results relating $\Z{\End R M}$ to $\End R {\T R M}$ are established in Section \ref{mainsection}. The remainder of the paper applies these results  in various settings. 

Section \ref{freeEnd} is concerned with modules that have $R$-free endomorphism rings.  We extend \cite[Theorem 3.1]{Vasc1}: 

\begin{Theorem}
 Let $R$ be a local Noetherian ring of depth $\leq 1$ and $M$ a finitely generated reflexive $R$-module. If $\End R M$ has a free summand then so does $M$. Therefore, $\End R M$ is a free $R$-module only if $M$ is a free $R$-module.
\end{Theorem}

Sections \ref{balanced}  and \ref{HWsection} apply our results to balanced and rigid modules; these are modules where $\Z{\End R M} = R$ and modules such that $\Ext 1 R M M =0$ respectively. For example, we prove the result below:

\begin{Theorem} \label{HW} Let $R$ be a one-dimensional Gorenstein local ring and $M$ a torsionfree faithful $R$-module. If  $M$ is rigid and $\Z {\End R M}$ is Gorenstein, then $M$ has a free summand.
\end{Theorem} 

When $M$ is an ideal, this result has been discovered independently, using different techniques, by Huneke, Iyengar and Wiegand in \cite{HIW}. These results support a conjecture of  Huneke and Wiegand (\cite[pgs 473-474]{HWrigidityoftor}) and we further prove that the conjecture holds for trace ideals; see Proposition \ref{HWconjtrace}.

\section{Trace Ideals} \label{properties}

In this section we collect those observations about trace ideals that are needed in subsequent sections. Throughout $R$ will  be a commutative ring and $M$ an $R$-module. Set  \[M^*:=\hom R M R,\] viewed as an $R$-module.  The \emph{trace map} is the map
\begin{equation}\label{theta}
\begin{tabular}{rccc}
$\vartheta_M$: & $\MO$ & $\lra $ & $R$\\
& $m\otimes \alpha$ & $ \longmapsto$ & $\alpha(m),$

\end{tabular}
\end{equation}

and the image of trace map is the {trace ideal} of $M$ since 
\begin{align*}
 \IM {\vartheta_M} &=\left \{\sum_i^n \alpha_i(m_i) | \alpha_i \in M^*, m_i \in M\right \}\\
 &= \sum_{\alpha \in M^* } \alpha (M)  \\
 &= \T R M .
\end{align*}

There is  a natural \emph{evaluation map}

\begin{center}
\begin{tabular}{cccc}
$\varepsilon_M:$&$M$ & $\lra $ & $ ({M^*})^*$\\
&$m$& $\mapsto$ & $\{ \psi \mapsto \psi(m)\}$\\
\end{tabular}
\end{center}

\begin{definition} The $R$-module $M$ is said to be \emph{torsionless} if $\varepsilon_M$ is an injection, and \emph{reflexive} if $\varepsilon_M$ is an isomorphism.  \end{definition}

When $M$ is reflexive,  Hom-Tensor adjointness yields the isomorphisms:
\begin{equation} \label{bigiso} 
\xymatrix @R=.1pc{
 \hom R M M \ar@[->][r]^{\cong} \ar@/^2pc/[rr]^{\vartheta}&{\hom R M {M^{**}}} \ar@[->][r]^{\cong} & {\hom R \MO R} \\
 \{f\} \ar@{|->}[r]& \{m \mapsto \varepsilon_M(f(m))\} \ar@{|->}[r]&  \{m \otimes \alpha \mapsto \alpha(f(m))\}.
}
\end{equation}

Notice, the trace map is the image of the identity endomorphism, that is \[\vartheta_M = \vartheta (\mbox{id}_M).\]  For ease of notation we write  $\vartheta_f$ for $\vartheta(f)$  given any non-identity $f$ in $\hom R M M$, so that 
\begin{equation}\label{phi}\vartheta_f(m\otimes \alpha) := \alpha(f(m))\end{equation}
for any $m \otimes \alpha$  in $\MO$.

As suggested in \cite{OrderIdeals},  one source of trace ideals is ideals of sufficiently high grade. 

\begin{definition}
Let $R$ be a Noetherian ring. The \emph{grade} of an ideal $I\subseteq R$, denoted $\grade I$, is the common length of the maximal $R$-sequences contained in $I$. 
\end{definition}

\begin{remark}
Grade is often calculated through the vanishing of certain Ext modules. In particular, using \cite[Theorem 1.2.5]{BandH} one has
\[\grade I = \mbox{min}\{ i | \Ext i R {R/I} R \not= 0\}.\]
\end{remark}

\begin{example}\label{grade} Let $R$ be a commutative Noetherian ring and $I$ an ideal with \[\grade I \geq 2.\]  Then $\T R I = I.$

Applying $\hom R ? R$ to the short exact sequence \[0 \lra I \lra R \lra R/I \lra 0\] yields  an exact sequence  \[0 \lra \hom R {R/I} R  \lra \hom R R R \lra \hom R I R \lra \Ext 1 R {R/I} R. \]

Since $\grade I \geq 2$ \[\Ext 1 R {R/I} R = 0  =\hom R {R/I} R\]and so all homomorphisms from $I$ to $R$ are given by multiplication by an element of $R$. The images of all such homomorphisms are contained in $I$; therefore $\T R I = I.$

\end{example}

\begin{remark} \label{tracepresmatrix}
We may calculate the trace ideal of a module from  its presentation matrix.  Indeed, suppose $A$ is a presentation matrix for a module $M$ and $B$ is a matrix whose columns generate the kernel of $A^*$, the transpose of $A$. Then there is an equality:
\[ \T R M = I_1(B);\]

where $I_1(B)$ is the ideal generated by the entries of $B$; see,   \cite[Remark 3.3]{Vascaffine}.
\end{remark}

The trace ideal of $M$ is the largest ideal (with respect to inclusion) generated by $M$ in the following sense:  

\begin{definition}\label{gen}
Let $M$ and $N$ be $R$-modules. We say $N$ is \emph{generated} by $M$ if  $N$ is the homomorphic image of a direct sum of copies of $M$. Said otherwise, there is an exact sequence
\[ M^{(\Lambda)} \lra N \lra 0 \]
for some index set $\Lambda$. 

An $R$-module $M$ is a \emph{generator} for the category of finitely generated $R$-modules (denoted $R$-mod), if $M$ is finitely generated and generates every finitely generated $R$-module.  Equivalently,  $M$ is a \emph{generator} if $R$ is a direct summand of $M^n$ for some $n \in \N$. 
\end{definition}

\begin{remark}\label{Mgenstrace} $M$ generates $\T R M$ and $M \otimes_R N$ for any $R$-module $N$.

\end{remark}

In the next proposition we collect the basic properties of trace ideals needed in this work.  Proofs are given for lack of a comprehensive reference. 

\begin{proposition}\label{trace properties} Let  $M$  be a finitely presented $R$-module. The following hold:
\begin{enumerate}[label=(\roman*)]
\item\label{trace generation} \label{trace isomorphic} If $M$ generates $N$, then $\T R M \supseteq \T R N $. 
\item\label{trace sums} $\T R {M\oplus N} = \T R M + \T R N$;
\item\label{trace R} $\T R M = R$ if and only if  $M$ generates all finitely generated $R$-modules. When $R$ is local, this is equivalent to $M$ having a nonzero free summand;
\item\label{trace contains} $I \subseteq \T R I $ for ideals $I$, with equality when $I$ is a trace ideal;
\item\label{trace shared} $\T R {\MO}= \T R {M} \subseteq \T R {M^*}$, with equality when $M$ is reflexive;
\item\label{trace dual}  $\End R {\T R M}= {\T R M}^* $;
\item\label{trace grade}  $\Ann {\T R M} = \Ann M$ when $M$ is reflexive. When, in addition, $R$ is Noetherian and  $M$ is faithful, then  $\T R M$ has positive grade;
\item\label{trace extend}  $\T R M  \otimes_R A =  \T { A} {M \otimes_R A}$ for any finitely presented $R$-module $M$ and any commutative flat $R$-algebra $A$. In particular, taking the trace ideal commutes with localization and completion.
\end{enumerate}
\end{proposition}

\begin{proof}
\ref{trace generation}: Any $R$-homomorphism $\alpha \in N^*$ can be composed with the surjection  \[M^{(\Lambda)} \overset{\gamma}{\lra} N \lra 0\]

for a given index set $\Lambda$. As $\IM {\alpha \gamma} \supseteq \IM \alpha$ it follows that $\T R N \subseteq \T R M$. 

\ref{trace sums}: This is clear.

\ref{trace R}:  Suppose $M$ is a generator in $R$-mod. Then, in particular, $M$ generates $R$. 
 By \ref{trace generation}, $\T R M \supseteq \T R R  = R$.
 
 If $\T R M = R$, then by Remark \ref{Mgenstrace} $M$ generates $R$. Since $R$ generates $R$-mod, $M$ too is a generator. This argument also works over noncommutative rings; see \cite[Theorem 18.8]{Lam}.

When $R$ is local, 
if $M = N \oplus R$, then $M$ clearly generates $R$. Now assume $\T R M = R$. There must exist $\alpha_i \in M^*$ and $m_i \in M$ such that \[1 = \displaystyle\sum_{i=1}^n \alpha_i(m_i).\] So  at least one $\alpha_i(m_i)$ is a unit. For such an $i$, the map $\alpha_i: M \lra R $ is surjective  and thus $M$ has a free summand.

\ref{trace contains}: For an ideal $I\subseteq R$, the inclusion map is an element of $I^*$ and therefore $I \subseteq \T R I$. In particular $\T R M \subseteq \T R {\T R M}$. The reverse inclusion holds because  $M$ generates $\T R M$; see \ref{trace generation}. 

\ref{trace shared}: Recall that $M$ generates $\MO$, and $\MO$, in turn, generates $\T R M$ (see the map $\vartheta_M$ from \ref{theta}), so one has 
\begin{align*}
\T R M &\supseteq \T R {\MO} & \text{by } \ref{trace generation}\\
&\supseteq \T R {\T R M} & \text{by }  \ref{trace generation}\\
&= \T R M & \text{by }  \ref{trace contains}
\end{align*}
 and  $\T R M = \T R {\MO}$.
 
Given that $M^*$ also generates $\MO$ one gets \[ \T R {M^*}\supseteq \T R {\MO} =\T R M .\]

When $M$ is reflexive, applying the above to $M^*$ one has \[ \T R M = \T R {M^{**}} \supseteq \T R {M^*} \supseteq \T R M .\]

\ref{trace dual}: Given $\alpha \in {\T R M}^*$,  one has \[\IM \alpha \subseteq \T R {\T R M} = \T R M.\] It follows that ${\T R M}^* = \End R {\T R M}$.

\ref{trace grade}:  Given $\psi \in \hom R {\MO} R$,  recall that \[\IM \psi \subseteq  \T R {\MO} = \T R M.\] This justifies the first of the following equalities,
\begin{align*}
\Ann {\T R M} &=  \Ann {\hom R \MO R} \\
& =   \Ann {\End R M} & \because M \text{ is reflexive} \\
& =   \Ann M 
\end{align*}

Assume, in addition, that $R$ is Noetherian and $M$ is faithful. This implies  that $\Ann {\T R M} = 0$, and so  $\T R M$ is not contained in any associated prime of $R$. By prime avoidance, $\T R M$ is not contained in the union of the associated primes of $R$.  Therefore $\T R M$ must contain a nonzerodivisor. 

\ref{trace extend}: Let $\lambda: R \lra A$ be a homomorphism of commutative rings with $A$ flat over $R$.
Since $A$ is flat,  the inclusion $\T R M \subseteq R$ yields the inclusion $$ 0 \lra \T R M \otimes _R A  \lra R \otimes_R A = A.$$

There are isomorphisms 
\begin{align*}
 \MO \otimes_R A  &\cong  \left(M \otimes_R A\right) \otimes_A \left(M^* \otimes_R A\right)\\
&\cong \left(M\otimes_R A\right) \otimes_A  \hom A {M \otimes_R A}{R\otimes_R A} ;
\end{align*}
where the second isomorphism holds because $M$ is finitely presented. 

Let $a,b \in A$, $\alpha \in M^*$ and $m \in M$. Together with the trace maps of $M$ as an $R$-module and $M \otimes_R A$ as an $A$-module,  we obtain a commutative diagram:
\begin{figure}[h]
\resizebox{13cm}{!}
{
$$\xymatrix@C=1em{ 
{\left( \MO \right)\otimes_R A}\ar@{->>}[rrr]^{\vartheta_M \otimes_R 1_A}   &&& \T R M  \otimes_R A \ar@{_{(}->}[dd] \\
& m \otimes_R \alpha \otimes_R ba  \ar@{|->}[r]&\alpha(m) \otimes_R ba   \ar@{|->}[d]&\\
\left(M \otimes_R A\right) \otimes_A \left(M^* \otimes_R A\right)\ar@{->}[dd]_{\cong} \ar@{->}[uu]^{\cong}&  (m \otimes_R a) \otimes_A (\alpha \otimes_R b)  \ar@{|->}[u] \ar@{|->}[d]& ab \lambda(\alpha(m) ) & A  \\
& (m\otimes_R a) \otimes_A b \cdot (\alpha \otimes_R 1_A)  \ar@{|->}[r]&  \alpha(m) \otimes_R ab \ar@{|->}[u] &\\
\left(M\otimes_R A\right) \otimes_A  \hom A {M \otimes_R A} {R\otimes_R A}  \ar@{->>}[rrr]^{\hspace{.6in} \vartheta_{(M\otimes_R A)}} &&&\T {R\otimes_R A} {M \otimes_R A} \ar@{^{(}->}[uu]
}$$
}
\end{figure}

 This demonstrates the desired equality as subsets of $A$:  \[ \T R M  \otimes_R A =  \T { A} {M \otimes_R A}\]
 
The left side representing the extension of the trace ideal of the $R$-module $M$,  to the ring $A$ and the right side being the trace ideal of  the $A$-module $M \otimes_R A$. 
\end{proof}

\section{Centers of Endomorphism Rings} \label{mainsection}
Let $R$ be a commutative ring and $M$ a finitely generated $R$-module. In what follows,  $\Z S$  denotes the center of a ring $S$ and $Q(S)$ denotes the total ring of quotients. We write $\beta_R(M)$ for the minimal number of generators of $M$.

In  Theorem \ref{main}, the first main result of this section, we construct an $R$-algebra isomorphism, $\sigma$, between $\End R {\T R M}$ and $\Z{\End R M}$ when $M$ is reflexive and faithful. We call on the properties of the trace ideal from Section \ref{properties} as well as results established over non-commutative rings by Suzuki in \cite{Suzuki1974}. For ease of reference, we include a proof of Suzuki's result below. 

 If $M$ is reflexive, there is a monomorphism from $\End R  {\T R M}$ to $\End R M$. Indeed, one may apply $\hom R {\,?\,} R $ to  the exact sequence \begin{equation}\label{surj} \MO \overset{\vartheta_M}{\lra} \T R M  \lra 0,\end{equation}

to obtain the top row of the diagram below.
\begin{equation}\label{inj1}
\xymatrix{
0 \ar@{->}[r] & \hom R {\T R M} R  \ar@{->}[d]^{=}_{\ref{trace properties} \ref{trace dual}}  \ar@{->}[r]^{\vartheta_M^* } &  \hom R \MO R  \ar@{->}[d]^{\vartheta^{-1} \text{ in }\ref{bigiso}}_{\cong}\\
& \End R {\T R M} \ar@{^{(}-->}[r]^{\sigma} & \End R M 
}
 \end{equation}
First we will show that the induced map $\sigma$ is an $R$-algebra monomorphism with image in the center of $\End R M$. Then we prove that when $M$ is also faithful, $\sigma$ is an isomorphism onto the center of $\End R M$. 

In what follows, we identify each $\alpha \in M^*$ with the map $\alpha : M \lra \T R M$.

\begin{lemma}\label{B} Let $M$ be a reflexive $R$-module.  The image of $\sigma$ in (\ref{inj1}) is the set 
  $$  \{ f\in \End R M : \exists ! \, \tilde{f}   \mbox{ s.t. the square below commutes} \Fa \alpha \in M^* \}$$
$$
 \xymatrix{
M \ar@{->}[d]^{\alpha}\ar@{->}[r]^f \null& M \ar@{->}[d]^{\alpha}\\
\T R M  \ar@{-->}[r]^{\tilde{f}} & \T R M 
} 
$$
\end{lemma}

\begin{proof} Let $B$ denote the set defined in the statement. 

 If  $f \in {B}$,  there exists a unique $\tilde{f} \in \End R {\T R M}$ such that  for any given $m \otimes \alpha \in \MO$ as in (\ref{phi}) and (\ref{theta}):  
\begin{align*}
\vartheta_M^* ( \tilde{f} )(m \otimes \alpha)& =  \tilde{f} \circ \vartheta_M(m \otimes \alpha) \\
&= \tilde{f} \alpha(m)  \\
&= \alpha (f(m))\\
&= \vartheta_f (m \otimes \alpha).
\end{align*}

It follows that  $\sigma(\tilde{f}) =\vartheta^{-1}(\vartheta_f)= f$ and so $f \in \IM \sigma$. 

On the other hand, say $\tilde{f} \in \End R {\T R M}$.  Then \[\vartheta_M^* (\tilde{f})\in \hom R \MO R \overset{\vartheta^{-1}}{\underset{\cong}{\lra}}  \hom R M M.\] In particular, there exists an  $ f \in \End R M$ such that, in the notation established above, $\tilde{f}\circ \vartheta_M  = \vartheta_M^* (\tilde{f})= \vartheta_f$. One gets
\[ \tilde{f}(\alpha(m)) = \tilde{f}\circ \vartheta_M(m \otimes \alpha) = \vartheta_f(m\otimes \alpha) = \alpha(f(m))\]
for all $m \otimes \alpha \in \MO$. That is, the square in the definition of $B$ commutes. 

Say $\alpha (g(m)) = \tilde{f}(\alpha(m)) = \alpha (f(m))$ for some other $g\in \hom R M M$ and all $m\otimes_R \alpha \in \MO$. Then $\alpha (f(m) -g(m)) = 0$ for all $\alpha \in M^*$ and  $m \in M$. Since $M$ is torsionless $f(m) = g(m)$ for all $m \in M$. Thus $\tilde{f}$ is unique to $f$. It follows that $\IM \sigma \subseteq {B}$
\end{proof}

\begin{proposition}\label{algebra map}
Let $M$ be a reflexive $R$-module. The map $\sigma$ in (\ref{inj1}) is a monomorphism of $R$-algebras with $\IM \sigma \subseteq \Z {\End R M}$. 

\end{proposition}

\begin{proof}
Take $f \in \IM \sigma$, $g \in \End R M$, $m \in M$ and $\alpha \in M^*$. Since $ \alpha g \in M^*$ and $g(m) \in M$, by Lemma \ref{B}  there exists $\tilde{f} \in \End R {\T R M}$ such that
\[(\alpha g)(f(m) )= \tilde{f} (\alpha g(m)) = (\alpha f) (g (m)). \]

This shows $\alpha(gf(m)-fg(m) )= 0$ for all $\alpha \in M^*$ and $m \in M$. As $M$ is torsionless, it follows that $gf(m) = fg(m)$ for all $m \in M$.

Further, $\sigma$ is an $R$-algebra homomorphism; for $\tilde{f}, \tilde{g} \in \End R {\T R M}$:
\[\tilde{f}\tilde{g} \circ \vartheta_M( m \otimes \alpha)  = \tilde{f}\tilde{g} (\alpha (m)) = \tilde{f} \alpha (g(m)) = \alpha(fg(m)) = \vartheta_{fg}(m \otimes \alpha)\]

That is, $\sigma (\tilde{f}\tilde{g}) = fg = \sigma(\tilde{f}) \sigma(\tilde{g})$. 
\end{proof}

\begin{corollary}\label{comm} Let $M$ be a reflexive $R$-module. Then $\End R {\T R M}$ is  a commutative ring.  \qed

\end{corollary}

\begin{remark}
Consider an ideal $I$ with $\grade I \geq 2$.  As a consequence of Example \ref{grade}, the ring of endomorphisms, $\End R I $, is $R$.  However, there do exist ideals for which $\End R I$ is not commutative.

For example, given a field $k$, the ideal $(x,y)$ in the ring $k[x,y]/(xy,y^2)$ has a noncommutative endomorphism ring. To see this, consider  the $f,g \in \End R I$ with

\begin{center}
\begin{tabular}{ccc}
$f(x) = x$ &\&  & $f(y) = 0$\\
$g(x) = y$ & \& & $g(y) =0$. 
\end{tabular}
\end{center}
Here $fg \not= gf$.
\end{remark}

Given that ideals are typically not reflexive, the preceding corollary and  remark suggest the following:
\begin{Question}
What are the necessary and sufficient conditions for an ideal to be the trace ideal of a reflexive module?
\end{Question}

Now we construct $\sigma^{-1}$.

The  definition below and the result that follows are from Suzuki \cite{Suzuki1974}. For ease of reference, we include a proof of the theorem. 
\begin{definition} \label{Itorsionless}
 Let $M$ and $U$ be $R$-modules. We say $M$ is $U$-torsionless if for every nonzero $m \in M$, there exists $\alpha \in \hom R M U$ with $\alpha(m) \not=0$. Equivalently the natural map 
 \[ M \lra \hom R {\hom R M U} U \] is injective.
\end{definition}

\begin{remark}\label{sigma'}

If $M$ generates $U$, then every element $u \in U$ may be written \[ u = \sum \phi_i(m_i)\]
for some $m_i \in M$ and  $\phi_i \in \hom R M U$.

Given $q \in Q : = \Z{\End  R M}, $  we hope to define an element $\bar{q} \in \bar{Q}: = \Z {\End R U}$ as follows:
\[\bar{q} (u) = \bar{q} \left( \sum \phi_i(m_i)\right) = \sum \phi_i(q(m_i)),\]

and to define a map 
\[
\xymatrix@R=.5em{
\sigma' :\Z{\End R M} \ar@{->}[r] & \Z{\End R U} 
}
\]
by $q \mapsto \bar{q}$.

The Theorem below, which combines Theorem 1 and Lemmas 2 and 3 from \cite{Suzuki1974}, establishes one case in which this map is well-defined. 

\end{remark} 

\begin{theorem}\label{Suzuki1}
Let $M$ and $U$ be $R$-modules. If $M$ generates $U$ and $U$ is $M$-torsionless then the map defined in Remark \ref{sigma'},  \[\sigma':\Z{\End  R M} \lra \Z{\End R U},\] is homomorphism of rings and is a monomorphism whenever $M$ is also $U$-torsionless.
\end{theorem}

\begin{proof}
We must establish that $\sigma'$, defined in Remark \ref{sigma'},  is a well-defined  ring homomorphism and further that \[\IM {\sigma'} \subseteq \Z{\End R U}. \]

\emph{ $\sigma'$  is well-defined: }

 Given two representations of $u \in U$ \[  \sum \phi'_i(m'_i) = u = \sum \phi_i(m_i),\] to show $\sigma'$ is well-defined is to show $\bar{q} \left( \sum \phi'_i(m'_i) - \sum \phi_i(m_i)\right)= 0$. It is enough to show that for arbitrary $\phi_i \in \hom R M U$, $m_i \in M$ and $q\in \Z {\End R M}$, \[\sum \phi_i(m_i) = 0 \mbox{ implies}  \sum \phi_i(q(m_i))=0.\]

Assume $\sum \phi_i(m_i) = 0$. Take any $d\in \hom R U M$ and note that $d\phi_i \in \End R M$ and so commutes with $q$. It follows that:
\begin{align*}
0 =qd\left( \sum \phi_i(m_i) \right) & =  \sum qd\left(\phi_i(m_i) \right) \\
&=\sum d\phi_i(q(m_i)) \\
  &= d\left( \sum \phi_i(q(m_i))\right).
\end{align*}

Recall, $U$ is $M$-torsionless by assumption. Since $d\left( \sum \phi_i(q(m_i))\right)=0$ for any \\$d\in \hom R U M$, it must be that $\sum \phi_i(q(m_i))= 0$.

It is clear that $\sigma'$ is an $R$-module homomorphism. 

 \emph{Further, $\sigma'$ is a ring homomorphism:}
\begin{align*}
\overline{q_1q_2} \left(\sum \phi_i(m_i) \right) & =\sum \phi_i(q_1q_2(m_i))\\
&=\bar{q_1} \left( \sum \phi_i(q_2(m_i))\right)\\
&=\bar{q_1} \bar{q_2} \left( \sum (\phi_i(m_i))\right).
\end{align*}

That is, given $q_i \in Q$, \[\sigma' (q_1 q_2) = \sigma'(q_1)\sigma'(q_2).\]

\emph{When $M$ is $U$-torsionless, $\sigma'$ is a monomorphism:}

Given any nonzero $q\in Q$ there exists $m\in M$ with $q(m) \not=0$. Since $M$ is $U$-torsionless, there also exists $\phi\in \hom R M U $ with $\phi(q(m)) \not=0$. That is, $\bar{q}(\phi(m)) \not=0$ implying $\bar{q} \not=0$ and therefore $\sigma'$ is a monomorphism.

\emph{The image of ${\sigma'}$ is a subset of $\Z{\End R U}$:}

Finally, notice  $\bar{q} \in \bar{Q}$  since for any  $f \in \End R U$ and $u \in U$ :
\begin{align*} \pushQED{\qed}
\bar{q}(f(u)) &=  \bar{q}\left(f\left(\sum \phi_i(m_i)\right)\right) \\
&=\bar{q} \sum f\phi_i(m_i)\\
&:=\sum f\phi_i(q(m_i))\\
&=f \left(\sum \phi_i(q(m_i)) \right)\\
&:= f \bar{q}\left(\sum \phi_i(m_i)\right) =  f\bar{q}(u)  \qedhere
\end{align*}
\end{proof}

\begin{corollary} \label{Suzukicor}
If  an $R$-module $M$ is torsionless and faithful,  there is a monomorphism of $R$-algebras \[\sigma': \Z{\End R M} \lra \End R {\T R M}.\]
\end{corollary}

\begin{proof}
Recall that $M$ generates $\T R M$. 

For each map $\alpha \in M^*$, we have $\IM {\alpha} \subseteq \T R M$. Therefore $M$ torsionless implies that $M$ is $\T R M$-torsionless.  Also, $M$ faithful implies \[\T R M \cap \Ann M = 0 .\] So for all $t\in \T R M$ there exists $m\in M$ with $tm \not=0$. The maps 
\begin{center}
\begin{tabular}{cccc}
$R$ & $\lra $ & $R$\\
$1$ & $\mapsto$ & $m$\\
\end{tabular}
\end{center}
restricted to $\T R M$ demonstrate that $\T R M$ is $M$-torsionless.

The desired monomorphism exists by Theorem \ref{Suzuki1}. 
\end{proof}

\begin{theorem} \label{main}
Let $R$ be a Noetherian ring. If $M$ is a finitely generated, faithful and reflexive $R$-module then the map \[\sigma: {\End R {\T R M}} \lra \Z {\End R M},\] in (\ref{inj1}), is an isomorphism of $R$-algebras.
\end{theorem}

\begin{proof}
By Proposition \ref{algebra map}, $\sigma : \End R {\T R M} \inj \Z {\End R M}$ is a monomorphism of $R$-algebras; in particular $\End R {\T R M}$ is commutative. 
Now by Corollary \ref{Suzukicor} there exists a monomorphism \[\sigma': \Z {\End R M} \inj \Z {\End R {\T R M}} = \End R {\T R M}\] via $\{f \mapsto f'\}$ where $f'$ is defined as follows: for any $x \in \T R M$, there exists $\alpha_i \in M^*$ and $m_i \in M$ with $x = \displaystyle \sum_{i=1}^n \alpha_i(m_i)$; set 
\[ f' (x) = f'\left(\displaystyle \sum_{i=1}^n\alpha_i(m_i)\right) :=\displaystyle \sum_{i=1}^n\alpha_i(f(m_i)).\]

Notice
\begin{align*}
 f' \circ \vartheta_M \left(\displaystyle \sum_{i=1}^n m_i \otimes \alpha_i \right)& =f' \left(\displaystyle \sum_{i=1}^n\alpha_i(m_i)\right)\\
&=\displaystyle \sum_{i=1}^n\alpha_i(f(m_i)) \\
&= \vartheta_f \left(\displaystyle \sum_{i=1}^n m_i \otimes \alpha_i \right)
\end{align*}

So $\sigma(f') = \vartheta^{-1}(\vartheta_f) =  f$ and for a given $f \in \Z{\End R M}$, \[\sigma \sigma' (f) = \sigma(f') = f .\] Since both $\sigma$ and $\sigma'$ are ring monomorphisms, they are inverse isomorphisms.
\end{proof}

\begin{remark}
The assumption that $M$ is faithful in Theorem \ref{main} ensures that $\T R M \cap \Ann M = 0$, which allows one to construct $\sigma'$; see Corollary \ref{Suzukicor}. When $M$ is reflexive, by Proposition \ref{trace properties} \ref{trace grade} one has $\Ann M = \Ann {\T R M}$.  Thus, \[ \T R M \cap \Ann M = \T R M \cap \Ann {\T R M}\]  and so this ideal must consist of nilpotent elements. It follows that $ \T R M \cap \Ann M = 0$ whenever $R$ is reduced. The proof of Theorem \ref{main}  therefore also proves following Corollary.
\end{remark}

\begin{corollary}
Let $R$ be a reduced Noetherian ring. If $M$ is a finitely generated, reflexive $R$-module then the map \[\sigma: {\End R {\T R M}} \lra \Z {\End R M},\] in (\ref{inj1}), is an isomorphism of $R$-algebras. \qed
\end{corollary}

The assumptions on $M$ in Theorem \ref{main} may also be relaxed given additional assumptions on $\T R M$;
the following lemma, for example, is an extension of Exercise 95 in \cite{Kaplansky1954} and is proved, in a noncommutative setting, in \cite[Theorem A.2. (g)]{MaxOrders}.
\begin{lemma}\label{Kaplansky} If $\T R M = R$, then $\Z{\End R M} \cong R$.
\end{lemma}
\begin{proof}
Since $M$ generates $\T R M =R$, there exists an $n \in \N$ and an $R$-module $N$ such that  $M^n \cong N \oplus R$. It follows that $M^n$ is faithful and therefore the map $R \lra \Z{\End R {M^n}}$ sending $r\in R$ to multiplication by $r$ is an injection. Consider the endomorphism, $\pi$, which projects $M^n$ onto $R$, that is $\pi(n +r ) =  r$ for any $n  + r \in N\oplus R$.  If $f \in \Z{\End R {M^n}}$, then 
\begin{align*} f(r)&= f\circ\pi(n+r)\\
& = \pi \circ f(n+r) \in R.\end{align*}

It follows that $f$ restricted to $R \subseteq M^n$ is an element of $\End R R$ and so is given by multiplication by some $x \in R$. Now for any $m \in N$, take the endomorphism $\gamma_m: N \oplus R \lra N \oplus R$ such that $\gamma_m(n+r) = rm$. One gets 
\begin{align*} 
f(m) &= f \circ \gamma_m (1) \\
&= \gamma_m \circ f(1) \\
&= \gamma_m (x) \\
&= x \gamma_m (1) = xm
\end{align*}

Altogether,  every $f \in \Z {\End R {M^n}}$ is given by multiplication by some element $x$ in $R$. That is to say $\Z {\End R {M^n}} = R$. It follows that $\Z{\End R M} \cong R$ because $\End R M$ and $\End R {M^n}$ are morita equivalent rings; see \cite[Corollary A.2]{LeuschkeCrepant} and \cite[18.42]{Lam}.
\end{proof}

\begin{remark} 
In general, the converse of Lemma \ref{Kaplansky} need not hold; Example \ref{grade} shows that $\End R I = R$ for all ideals with $\grade I \geq 2$.  We use Theorem \ref{main} to prove the converse for reflexive modules over local Noetherian rings of depth less than or equal to one; see Proposition \ref{balanceddim1}.
\end{remark}

The following identifications are well-known; see Exercise 4.31 in \cite{WLCMR}.

\begin{lemma}\label{WiegandLeuschke}
Suppose $I\subset R$ is an ideal containing a nonzerodivisor $x$. Then the map \[ f\lra \dfrac{f(x)}{x}\] identifies $\hom R I R $ and $\hom R I I$ with  \[\left\{ \dfrac{a}{b} \in Q(R) | \dfrac{a}{b} \cdot I \subseteq R \right\} \mbox{ and } \left\{ \dfrac{a}{b} \in Q(R) | \dfrac{a}{b} \cdot I \subseteq I \right\},\] respectively, and this identification is independent of the choice of $x$. \qed

\end{lemma}

\begin{remark} \label{identifyend}
Recall $\T R M ^* = \End R {\T R M}$. As a consequence of Lemma \ref{WiegandLeuschke}, $\End R {\T R M}$ is commutative whenever $\T R M$ contains a nonzerodivisor. This is the case, for example, when the  $R$-module $M$ is finitely generated, reflexive and faithful; see Proposition \ref{trace properties} \ref{trace grade}.  

When $\T R M$ has positive grade we henceforth identify $\End R {\T R M}$ with
\[ \left\{ \dfrac{a}{b} \in Q(R) | \dfrac{a}{b} \T R M \subseteq \T R M \right\}.\]

\begin{definition}
Let $M$ be an $R$ module. We say $M$ is \emph{torsionfree} provided that given any  nonzerodivisor $r\in R$ and $m \in M$,  if  $rm =0$ then $m=0$.  
\end{definition}

\begin{corollary} \label{inQ}
Let $R$ be a Noetherian ring. If $M$ is a finitely generated, faithful, reflexive $R$-module,  then
\[\Z {\End R M} =  \End R {\T R M}\]
as subsets of $Q(R)$, where equality is understood in the following sense: there is a bijection between the sets given by $\sigma$ from Theorem \ref{main} and for any  $a/b \in \End R {\T R M}$ the endomorphism $\sigma (a/b)$ is multiplication by $a/b$. That is, $aM \subseteq bM$ and $\sigma(a/b)(m) = (am)/b\in M$ for each $m\in M$.

\end{corollary}

\end{remark}

\begin{proof}
By  the hypotheses on  $M$ and Remark \ref{identifyend}, we identify $\End R {\T R M}$  with
\[ \left\{ \dfrac{a}{b} \in Q(R) | \dfrac{a}{b} \T R M \subseteq \T R M \right\}.\]
 Theorem \ref{main} provides a bijection between $\End R {\T R M}$ and $\Z{\End R M}$,  so that a given $f\in \Z{\End R M}$ is $\sigma({a/b})$ for some unique $a/b \in \End R {\T R M}$. 
 
 By definition $\sigma(1) = id_M$. Therefore,
\begin{align*}
 b \left(\sigma \left( \dfrac{a}{b}\right) ( m)\right) &=  \left( b \sigma \left( \dfrac{a}{b}\right)\right) ( m)\\
&= a\sigma(1)(m) \\
&= a( \mbox{id}_M (m) )\\
&= am.
\end{align*}
It follows that  $am = bn$ for some $n  \in M$. Since $b$ is a nonzerodivisor and $M$ is torsionfree \cite[Exercise 1.4.20 (a)]{BandH}, \[b\left( \sigma(a/b) (m)   - n\right) =0\]
implies $\sigma(a/b) (m)=n = (am)/b$. \end{proof}

When  $\T R M$ has positive grade but $M$ is not necessarily reflexive and faithful  (for example any module over a domain) we shall prove a result similar to Theorem \ref{main}, this time relating $\End R {\T R M}$ to $\Z {\End R {M^*}}$. 

Note, for each $\alpha \in M^*$, one has $ \IM \alpha \subseteq \T R M$. Thus there is a natural map of $R$-algebras $\End R {\T R M} \lra \End R {M^*}$ because $M^*$ is always a left $\End R {\T R M}$-module via the composition:
\[ M \overset{\alpha}{\lra} \T R M \overset{h}{\lra} \T R M  \subseteq R.\]

In particular the surjection (\ref{surj})  yields a second commutative diagram
\begin{equation}\label{inj2}
\xymatrix{
0 \ar@{->}[r] & \hom R {\T R M} R  \ar@{->}[d]^{=}_{\ref{trace properties} \ref{trace dual}}  \ar@{->}[r]^{\vartheta_M^* } &  \hom R \MO R  \ar@{->}[d]_{\cong}\\
& \End R {\T R M} \ar@{^{(}-->}[r]^{\rho} & \End R {M^*} 
}
 \end{equation}

where the isomorphism on the right follows from Hom-Tensor adjointness, so that an endomorphism $g$ in $\End R{M^*}$ is the image of $\psi_g$ in $\hom R {\MO} R$ such that \[\psi_g ( m \otimes \alpha ) = g(\alpha) m.\] The induced map $\rho$ sends $h \in \End R {\T R M}$ to composition by $h$.  To illustrate,  recall $ h\vartheta_M =\vartheta^*_M(h)  \in \hom R {\MO } R$ so that for each $m$ in $M$ and $\alpha$ in $M^*$ there exists some $\psi_g $ in $\hom R {\MO} R$ and $g$ in $\End R {M^*}$ such that 
\[ (h \alpha)(m) = h (\alpha (m)) = h \vartheta_M (m \otimes \alpha ) = \vartheta^*_M(h)(m \otimes \alpha)= \psi_g ( m \otimes \alpha ) = g(\alpha) (m).\] 

That is, $\rho(h) = g$ and $h\alpha = g(\alpha).$

Again we claim this injection, $\rho$,  is an $R$-algebra isomorphism from  $\End R {\T R M}$ onto the center of $\End R {M^*}$.

\begin{remark} \label{Mdualtorsionless}
It is easy to show that for a finitely generated $R$-module $M$, the $R$-dual, $M^*$, is both torsionfree and torsionless. If $\T R {M^*}$ contains a nonzerodivisor, $M^*$ is also faithful. 
\end{remark}

\begin{proposition} \label{image2} When $\T R M$ has positive grade, the map $\rho$ in (\ref{inj2}) is a monomorphism of R-algebras with \[\IM {\rho} \subseteq \Z{\End R {M^*}}.\]

\end{proposition}
\begin{proof}
Recall the identification from Lemma \ref{WiegandLeuschke}. Taking ${a}/{b}$ in $ \End R {\T R M}$, a map $\alpha$ in $M^*$ and an endomorphism $g$ in $\End R {M^*}$, we have:
\[ b \left( \dfrac{a}{b}g(\alpha) - g\left(\dfrac{a}{b} \alpha\right) \right)= 0.\]

Since $b$ is a nonzerodivisor and $M^*$ is torsionfree, we conclude \[\dfrac{a}{b} \cdot g(\alpha) = g\left(\dfrac{a}{b} \cdot \alpha\right).\]

Recall the injection $\rho: \End R {\T R M} \inj  \End R {M^*}$ constructed in (\ref{inj2}) sends  ${a}/{b}$ in $\End R {\T R M}$ to composition by ${a}/{b} $. Since  \[\dfrac{a}{b}g(\alpha) = g\left(\dfrac{a}{b} \alpha\right),\] for all $g \in \End R {M^*}$ and ${a}/{b} \in \End R {\T R M}$,  the image of $\rho$ is in the center of $\End R {M^*}$.  

It is clear that \[\rho\left(\dfrac{a}{b} \cdot \dfrac{a'}{b'}\right) (\alpha) =\dfrac{a}{b} \cdot \dfrac{a'}{b'} (\alpha)  = \rho \left(\dfrac{a}{b}\right) \left( \dfrac{a'}{b'} \alpha \right)   =  \rho \left(\dfrac{a}{b}\right) \rho\left( \dfrac{a'}{b'}\right) (\alpha ) \] 

and therefore $\rho$ is a monomorphism of $R$-algebras.
\end{proof}

\begin{remark} \label{torsionlesstransfer}
Let $M$ be an $R$-module and $I\subseteq J$ ideals. If $I$ contains a nonzerodivisor and $M$ is $J$-torsionless, then $M$ is $I$-torsionless. That is, all nonzero $m$ in $M$ have a nonzero image in $I$ under some $\alpha \in M^*$ which depends on $m$; recall Definition \ref{Itorsionless}. This follows from the fact that for all nonzero $m \in M$ there exists $\alpha \in \hom R M J$ with $\alpha(m) \not=0$ and therefore for any nonzerodivisor $x\in I$,  the product $x\alpha(m) \not=0$ with $x\alpha \in \hom R M I$. We use this observation in the proof of the following result.
\end{remark}
\begin{theorem}\label{main2}
Let $R$ be a Noetherian ring. Suppose $M$ is a finitely generated $R$-module such that $\T R M$ has positive grade. Then the map   \[\rho: \End R {\T R M} \lra \Z{\End R {M^*}}.\]
in (\ref{inj2}) is an $R$-algebra isomorphism.
\end{theorem}
\begin{proof}

Recall, $M^*$ is $\T R {M^*}$-torsionless, faithful and generates $\T R M$. Since $\T R M$ contains a nonzerodivisor and $\T R {M^*} \supseteq \T R M$, $M^*$ is $\T R {M}$-torsionless and since $M^*$ is faithful, $\T R M$ is $M$-torsionless; see Remarks \ref{Mdualtorsionless}, \ref{torsionlesstransfer} and the proof of Corollary \ref{Suzukicor}. By Theorem \ref{Suzuki1} there is an injection \[\Z{\End R {M^*}} \inj \End R {\T R {M}}.\] \[ q \mapsto \bar{q}\]
By Proposition \ref{image2}, there is another injection \[\rho: \End R {\T R M} {\inj} \Z{\End R {M^*}}\] 

which sends $a/b \in \End R {\T R M}$ to composition by $a/b$.

These maps are inverses. Indeed, given $t \in \T R M$, there exist $\alpha_i \in M^*$ and $\varphi_i \in \hom R {M^*} {\T  R M}$ such that \[ t = \sum_i \varphi_i(\alpha_i).\]  For $a/b \in \End R {\T R M}$, 
\begin{align*}
 \overline{\rho\left(\dfrac{a}{b}\right)} (t) &= \overline{\rho\left(\dfrac{a}{b}\right)} \left(\sum_i \varphi_i(\alpha_i) \right) \\
 & = \sum_i \varphi_i\left(\rho\left(\dfrac{a}{b}\right)\alpha_i\right)\\
 & = \sum_i \varphi_i\left(\dfrac{a}{b} \alpha_i\right)\\
 &=\dfrac{a}{b}\left(\sum_i \varphi_i\left( \alpha_i\right)\right) = \dfrac{a}{b} (t); 
\end{align*}

the penultimate equality following from $\T R M$ being torsionfree. That is, 
\[b\cdot \left[ \dfrac{a}{b}\left(\sum_i \varphi_i\left( \alpha_i\right)\right) - \sum_i \varphi_i\left(\dfrac{a}{b} \alpha_i\right)  \right]= 0 \]
implies $\sum_i \varphi_i\left({a}/{b} \cdot \alpha_i\right) = {a}/{b} \cdot\left(\sum_i \varphi_i\left( \alpha_i\right)\right) $. We conclude that $ \overline{\rho\left({a}/{b}\right)}   = a/b$, that is, the composition of the injections is identity.
\end{proof}

\begin{remark} When $M$ is reflexive and faithful, $\T R M$ contains a nonzerodivisor and $\T R {M^*} = \T R M$.  By Theorem  \ref{main2} \[\Z{\End R {M^{**}}} \cong \End R {\T R {M^*}}.\]
In other words,  \[\Z{\End R M} \cong \End R {\T R M}.\]
This is another proof of Theorem \ref{main}.

\end{remark}

\begin{corollary} \label{alliso}
Let $R$ be a Noetherian ring. Suppose $M$ is a finitely generated $R$-module such that $\T R M$ has positive grade.  Then \[ \End R {\T R {M^*}} \cong \End R {\T R M} \cong \Z{\End R {M^*}}\]
\end{corollary}

\begin{proof}

Recall $\T R {M^*} \supseteq \T R M$. Since $\T R M$ contains a nonzerodivisor, the module $\T R {M^*}/ \T R M$ is annihilated by that nonzerodivisor and is therefore torsion. Hence the inclusion $\T R M \subseteq \T R {M^*}$ yields an inclusion \[\hom R {\T R {M^*}} R \subseteq \hom R {\T R M} R\] as subsets of $Q(R)$. That is, \[ \End R {\T R {M^*}} \subseteq \End R {\T R M}.\]

Note, $M^*$ is torsionless and faithful, and also generates $\T R {M^*}$. Therefore, by Corollary \ref{Suzukicor}  and Remark \ref{identifyend} there is an injection \[\Z{\End R {M^*}} \inj \Z{\End R {\T R {M^*}}} = \End R {\T R {M^*}}.\]

By Theorem \ref{main2} there is  now a sequence of monomorphisms whose composition is identity on $\End R {\T R {M^*}}$:
\[
\pushQED{\qed}
\End R {\T R {M^*}} \subseteq \End R {\T R M} \overset{\rho}{\inj} \Z{\End R {M^*}} \inj \End R {\T R {M^*}}.
\qedhere
\] \end{proof}

\begin{corollary} \label{identify}
Let $R$ be a commutative Noetherian ring. Suppose $M$ is a finitely generated $R$-module such that $\T R M$ has positive grade.  Then $ \End R {\T R {M}}$,  $\End R {\T R {M^*}}$ and $\Z{\End R {M^*}}$ may be identified as $R$-submodules of the total ring of quotients $Q(R)$, and then one has \[\Z{\End R {M^*}} = \End R {\T R M} =  \End R {\T R {M^*}}\] as subsets of $Q(R)$.
\end{corollary}

\begin{proof}
As in Corollary \ref{inQ}, this identification is an easy consequence of the definition of the maps established in Corollary \ref{alliso}, and Lemma \ref{WiegandLeuschke}. A proof proceeds like the proof of Corollary \ref{inQ}. 
\end{proof}

\section{Projective Endomorphism Rings} \label{freeEnd}

Let $R$ be a commutative ring. In this section and the next we use  results from Section \ref{mainsection} to relate the properties of $\End  R M$ to those of $M$.  Below, we concentrate on when $\End R M$ is a free $R$-module. 

The following result is well-known: 

\begin{lemma} \label{dualR} Let $R$ be a local  Noetherian ring  of depth $\leq 1$ and $I \subseteq R$ an ideal. If $I^* = R$ then $I=R$.
\end{lemma}
\begin{proof}
Applying $\hom R ? R$, to $0 \lra I \lra R \lra R/I \lra 0$ yields  an exact sequence  \[0 \lra \hom R {R/I} R \lra \hom R R R \overset{i}{\lra} \hom R I R \lra \Ext 1 R {R/I} R \lra 0.\] 

For $r\in \hom R R R$, the map $i(r) \in \hom R I R$ is the restriction of multiplication by $r$ to $I$. By assumption $I^* =R$,  hence $i$ is an isomorphism, and \[ \hom R {R/I} R = 0 =  \Ext 1 R {R/I} R.\] This implies grade $I \geq 2$,  which is not possible given depth $R \leq 1$. We conclude that $R/I = 0$, that is, $I = R$.
\end{proof}

 \begin{theorem} \label{endtrace} \label{endtracecor}
 Let $R$ be a local Noetherian ring of depth $\leq 1$ and $M$ a finitely generated reflexive $R$-module. If $\End R M$ has a free summand then so does $M$. Therefore, $\End R M$ is a free $R$-module only if $M$ is a free $R$-module.
 \end{theorem}
 
 \begin{proof}
 Since $\End R M$ has a free summand, $\End R M$ is faithful. Therefore $M$ is faithful in addition to being reflexive, and hence $\T R M $ has positive grade by Proposition \ref{trace properties} \ref{trace shared},\ref{trace grade}. Then, as $R$-submodules of $Q(R)$

\begin{align*}
\T R M^* &= \End R {\T R M} & \text{ by Proposition } \ref{trace properties} \ref{trace dual}\\
 &= \End R {\T R {\MO}}  & \because \T R M = \T R {\MO} \\
 &= \End R {\T R {(\MO)^*}}  &\text{ by Corollary } \ref{identify} \\
 & =\End R {\T R {\End R M}}  & \because M \text{ is reflexive }\\
 & = R & \because \T  R {\End R M} = R.
\end{align*}

Since depth $R \leq 1$, Lemma \ref{dualR} applies and yields \[ \T R M  = R,\] and hence $M$ has a free summand.   

We write $M = N \oplus R$ for some $R$-module $N$. If $\End R M$ is a free $R$-module, then for some $n \in \mathbb{N}$ 
\[ R^n \cong \End R M \cong \hom R {N \oplus R} M \cong M \oplus \hom R N M.\]

As a direct summand of a free module, $M$ is projective and therefore free. 
\end{proof}

\begin{definition} Let $R$ be a local Noetherian  ring. A nonzero $R$-module $M$ is called \emph{maximal Cohen-Macaulay} (MCM) provided the depth of $M$ is equal to the Krull dimension of $R$.  
\end{definition}

\begin{definition} \label{gorenstein}
Let $R$ be a Noetherian local ring. Then $R$ is called Gorenstein provided $\Ext i R k R =0$ for some $i > \dim R$, where $k$ is the residue field of $R$; see \cite[Theorem 18.1]{Mats}. 

A Noetherian ring, $R$, is called Gorenstein provided its localization at every maximal ideal is a Gorenstein local ring. 
\end{definition}
The next result is \cite[Theorem 3.1]{Vasc1}.

\begin{corollary} 
Let $R$ be a one-dimensional Gorenstein ring and $M$ a
finitely generated $R$-module. If $\End R M$ is  projective then the $R$-module $M$ is projective.\end{corollary}

\begin{proof}
We may assume $R$ is local and therefore $\End R M$ is free. Since \[\Ass {\End R M} = \Supp (M) \cap \Ass M = \Ass M,\] the maximal ideal is not an associated prime of $M$. It follows that $M$ is MCM and being MCM over a Gorenstein ring,  $M$ is also reflexive; see \cite[Corollary 2.3]{Vasc1}. Then by Theorem \ref{endtracecor}, $M$ is free. 
\end{proof}

\section{Balanced Modules} \label{balanced}

In this section, let $R$ be a commutative Noetherian ring and let $M$ be a finitely generated $R$-module.
\begin{definition}
Let $R \lra\Z{ \End R M}$ be the natural map from $R$ to the center of $\End R M$ where $r\in R$ is sent to multiplication by $r$. The $R$-module $M$ is \emph{balanced} when this map is an isomorphism. 
\end{definition}

 We discuss the implications of balancedness for reflexive modules when $R$ has depth less than or equal to one. Over rings of arbitrary depth, we establish a type of purity theorem for balancedness; this property need only be checked at primes of grade less than or equal to one.

The following proposition approaches a converse to Lemma \ref{Kaplansky}. The proof of the proposition adapts ideas from Vasconcelos \cite{Vasccommend}.   

\begin{proposition}\label{balanceddim1} Let $R$ be a local  ring  of depth $\leq 1$  and $M$ a finitely generated reflexive  $R$-module. Then $M$ is balanced if and only if $M$ has a free summand.
\end{proposition}

\begin{proof}
If $M$ has a free summand, then  $M$ is balanced (apply the proof of Lemma \ref{Kaplansky} for $n=0$).

Now, suppose that $M$ is balanced. Then $M$ is also faithful since:
\begin{align*}
\Ann M  & = \Ann {\End R M}\\
& \subseteq \Ann {\Z{\End R M}} \\
&= \Ann R = 0.
\end{align*}

 By Corollary \ref{inQ}, $ \End R {\T R M} = \Z{\End R M} = R $. That is to say, $\T R M ^* =R$. By Lemma \ref{dualR} we conclude that $\T R M = R$. Therefore $M$ has a free summand; see Proposition \ref{trace properties} \ref{trace R}. \end{proof}
 
 The following is a corollary of Theorem \ref{endtrace} and Proposition \ref{balanceddim1}.
 \begin{corollary}
 Let $R$ be a local  ring  of depth $\leq 1$ and $M$ a finitely generated reflexive $R$-module. The following are equivalent:
 \begin{enumerate}[label=(\roman*)]
 \item $\End R M$ has a free summand;
 \item $M$ has a free summand;
 \item $M$ is balanced.    \qed
 \end{enumerate}
 \end{corollary}

\begin{definition}  We say a property of  an $R$-module $M$ holds \emph{in codimension n} if the $R_{\mathfrak{p}}$-module $M_{\mathfrak{p}}$ holds the property for all $ \mathfrak{p} \in \text{Spec } R$ such that grade $\mathfrak{p} \leq n$.  

Given that if $M$ has a free summand  then $M$ is balanced (Remark \ref{Kaplansky}), the following proposition extends \cite[Proposition 2.1]{Vasccommend}. 

\begin{proposition} \label{balancedpuritytorsionless}
Let $M$ be a finitely generated, torsionless and faithful $R$-module. If $M$  has a free summand in codimension one, then $M$ is balanced. 
\end{proposition}
\begin{proof}
Trace ideals behave well under localization by Proposition \ref{trace properties} \ref{trace extend}. Thus, by hypothesis,  ${\T R M}_{\mathfrak{p}}$ is free for all  prime ideals $\mathfrak{p}$  with depth $R_{\mathfrak{p}} \leq 1$. 

There are injections \begin{equation}\label{sandwich} R \inj \Z{\End R M}  \inj \End R {\T R M} ; \end{equation}
the first  one holds because $M$ is faithful, the second because $M$ is also torsionless; see Corollary \ref{Suzukicor}. Note, an $r\in R$ is sent to multiplication by $r$ under the first map and the composition.  

Set $C: = \End R {\T R M}$ and consider the exact sequence induced by the compositions of  injections in (\ref{sandwich}):
\begin{equation}  \label{centerses}
 0  \lra R \lra C \lra X \lra 0. 
\end{equation}

Notice $X_{\mathfrak{p}}= 0$ for all  prime ideals $\mathfrak{p}$  with grade $\mathfrak{p} \leq 1$. Therefore $X$ has a nonzero annihilator,  say $I$, with grade $ I \geq 2$. Apply $\hom R {R/I} ?$ to (\ref{centerses}) to get an exact sequence
\[ 0 \rightarrow  \hom R {R/I} R \rightarrow \hom R {R/I } C \rightarrow \hom R {R/I} X \rightarrow \Ext 1 R {R/I} R .\]

First,  $ \Ext 1 R {R/I} R =0$ by \cite[Thm 1.2.5]{BandH}. Second,  as an ideal $\T R M$ is torsionfree, hence $C$ is torsionfree. Recall, also, that since $I$ contains a nonzerodivisor, $R/I$ is torsion.  Therefore $\hom R {R/I } C  = 0$. Together, these force $ \hom R {R/I} X =0$. However, as $I = \Ann {X}$,  this implies $X=0$. That is to say, the inclusion $R \inj \End R {\T R M}$ is an equality. 

Then (\ref{sandwich}) reads:   \[R \inj \Z{\End R M} \inj R,\] and since both the first injection and  the composition send $r \in R$ to multiplication by $r$:
\[\pushQED{\qed}
{{\Z{\End R M}} = R.}
\qedhere\] 
\end{proof}

\end{definition} 
The following is a corollary of Propositions \ref{balanceddim1} and  \ref{balancedpuritytorsionless}.
\begin{corollary} \label{balancedpurity}
 Let $M$ be a finitely generated, reflexive and faithful $R$-module. The following are equivalent:
 \begin{enumerate}[label=(\roman*)]
 \item  $M$ is balanced in codimension one;
 \item  $M$ is balanced.  \qed
 \end{enumerate}
\end{corollary} 

\section{One-Dimensional Gorenstein Rings} \label{HWsection}

In this section we apply the results of the previous sections to rigid modules over Gorenstein rings of dimension one. In what follows, $R$ is a commutative ring. 

\begin{definition}
An  $R$-module, $M$, is called \emph{rigid} if $\Ext 1R M M = 0$. 
\end{definition}

\begin{remark} \label{HJ}
Let $R$ be a one-dimensional Gorenstein local ring. If $M$ is maximal Cohen-Macaulay then $\Ext 1 R M M  = 0$ implies $\MO$ is maximal Cohen-Macaulay; see \cite[Theorem 5.9]{HJsymmetry}. The converse holds when $\Ext 1 R M M$ has finite length, for example, when $M$ is free on the punctured spectrum.
\end{remark}

\begin{lemma} \label{same} \label{switchring}
 Let $S$ be a ring such that $R \subseteq S \subseteq Q(R)$ and suppose $M$ and $N$ are $S$-modules. If $N$ is torsionfree as an $R$-module then $\hom R M N  = \hom S M N$. If $M\otimes_R N$ is torsionfree as an $R$-module, then $M \otimes_R N  = M \otimes_S N$. \qed
\end{lemma}

\begin{theorem} \label{HW} Let $(R, \mathfrak{m})$ be a one-dimensional  Gorenstein local ring  and $M$ a finitely generated, torsionfree and faithful $R$-module. If  M is rigid and $\Z {\End R M}$ is Gorenstein, then $M$ has a free summand.
\end{theorem} 

\begin{proof}
The hypotheses are stable under completion. In particular,  ${M}$  and its $\mathfrak{m}$-adic completion $\widehat{M}$ are reflexive and faithful over $R$ and $\widehat{R}$ respectively. Therefore by Theorem \ref{main} and Proposition \ref{trace properties} \ref{trace extend}
\begin{align*}
\Z{\End {\widehat{R}}  {\widehat{M}}  } &\cong\End {\widehat{R}} {\T {\widehat{R}} {\widehat{M}}} \\
&\cong \End R {\T R M} \otimes_R {\widehat{R}}\\
&\cong \Z{\End R M} \otimes_R \widehat{R} \,\
\end{align*}
as $R$-algebras. The module $\End R {\T R M} \otimes_R \widehat{R}$  is the completion of  the ring $\End R {\T R M}$ with respect to its Jacobson radical.  It follows that $\Z{\End {\widehat{R}}  {\widehat{M}}  }$ is Gorenstein; see Definition \ref{gorenstein}.

Thus we may assume that $R$ and $\Z{\End R M}$ are complete.

We write $C$ for $\End {{R}} {\T {{R}} {{M}}}$ and note that $R \subseteq C$ is a finite extension and hence $C$ is a semi-local complete one-dimensional Gorenstein ring. Because $C$ is semi-local and complete, it decomposes as a product $C \cong \prod_{i=1}^{n} C_i$ where $C_i$  are complete local Gorenstein rings; see \cite[Thm 8.15]{BandH}. Thus
\begin{align*}
\T {{R}} {{M}} & \cong  \hom {{R}} {\hom {{R}} {\T {{R}} {{M}}} {{R}}} {{R}}\\
& \cong \hom {{R}} C {{R}}\\
& \cong  \hom {{R}} { \prod_{i=1}^{n}C_i} {{R}}\\
& \cong \prod_{i=1}^{n} \hom {{R}} {C_i} {{R}} \\
& \cong\prod_{i=1}^{n} C_i\\
& \cong C. 
\end{align*}

The first isomorphism follows from the reflexivity of $\T R M$  and $\hom {{R}} {C_i} {{R}} \cong C_i$ by \cite[Thm 3.3.7]{BandH}. 

Let $t = \beta_R(M^*)$.  Since $C \subseteq Q({R})$, Lemma \ref{switchring} applies and  the following $R$-homomorphisms are  also $C$-homomorphisms:  \[{M}^{t} \surj \T {{R}} {{M}} \overset{\cong}{\lra} C.\]
This is a surjective map from the $C$-module $M^t$ onto $C$. It follows that ${M}^{t} \cong N \oplus C$ for some $C$-module $N$.  One has
\[  \Ext 1 {{R}} C C \subseteq \Ext 1 {{R}} {{M}^{t}} {{M}^{t}} =0\]

implying $C\otimes_{{R}} \hom {{R}} C {{R}}$ is torsionfree; see Remark \ref{HJ}. Using Lemma \ref{same}, as $C$-modules, and therefore also as ${{R}}$-modules, one has : \[ C \otimes_R C^* = C \otimes_C C^* \cong C^* .\]

Note that, $\beta_R(C\otimes_RC^*) = \beta_R(C)\beta_R(C^*)$. The isomorphism $C\otimes_R C^* \cong C^*$ implies that $C$ is a cyclic  $R$-module. Moreover, ${R} \subseteq C$ and $C$ is $R$-torsionfree. Therefore $C \cong  {R}$.  Since $C = \T {{R}} {{M}} ^*$  and $\T {{R}} {{M}}$ is reflexive, one has  $\T {{R}} {{M}}\cong R$ and therefore ${\T R M} = {R}$ by Proposition \ref{trace properties} \ref{trace generation} and \ref{trace contains}. It follows, by Proposition \ref{trace properties}  \ref{trace R}, that $M$ has a free summand. 
\end{proof}

\begin{corollary} 
Let $R$ be a  d-dimensional local ring that is Gorenstein in codimension one. Let $M$  be a finitely generated torsionfree faithful $R$-module. If  $M$ is rigid and $\Z {\End R M}$ is Gorenstein in codimension one, then $M$ is balanced.
\end{corollary}
\begin{proof}
This is a consequence of Theorem \ref{HW} and Corollary \ref{balancedpurity}. 
\end{proof}

\begin{remark} The ring $\Z {\End R M}$ being Gorenstein does not, by itself, imply that $M$ has a free summand. For example, suppose $R$ is a one-dimensional commutative domain with finitely generated integral closure (for example, $R$ is complete) such that every ideal of $R$ is two-generated. Then every ring between $R$ and its integral closure is Gorenstein. In particular,  $\End R I$ is Gorenstein for each ideal $I$; see \cite[Section 7]{Bassubiquity}, \cite{SallyVasc2}. 
\end{remark}

However, over a one-dimensional Gorenstein local domain with $M$ a torsionfree module,  it is conjectured that rigidity of $M$  (equivalently $\MO$ is torsionfree) is sufficient to ensure $M$ is free: 
\begin{conjecture}\label{HWconj} (Huneke and Wiegand \cite[pgs 473-474]{HWrigidityoftor}) Let $R$ be a Gorenstein  local domain of dimension one and $M$ a nonzero finitely generated torsionfree $R$-module, that is not free. Then $M \otimes_R M^*$ has a nonzero torsion submodule.
\end{conjecture}

\begin{proposition} \label{HWconjtrace}
Conjecture \ref{HWconj} is true for any ideal isomorphic to a trace ideal. 
\end{proposition}
\begin{proof}
It is enough to prove the proposition for trace ideals. 
We write $C$ for $\End R {\T R M}$.  If $\T R M \otimes_R {\T R M}^*$ is torsionfree, then 
\[ \T R M \otimes_R {\T R M}^* =  \T R M \otimes_C C \cong \T R M.\]
The final map is also an $R$-isomorphism, implying ${\T R M}^*$ is cyclic over $R$.  Since ${\T R M}^*$ is torsionfree and $R$ is a domain, we have ${\T R M}^* \cong R$. Finally since $\T R M$ is reflexive (a nonzero ideal over a one-dimensional Gorenstein ring; see  \cite[Theorem 6.2]{Bassubiquity} \cite[Theorem A.1]{Vasc1}), $\T R M \cong R$. This  implies $\T R M =R$ by Proposition \ref{trace properties} \ref{trace generation} and  \ref{trace contains}. \end{proof}

\begin{Question}
Over a Noetherian ring of depth one, which  ideals are isomorphic to a trace ideal?
\end{Question}

For a local ring $(R, \mathfrak{m}, k)$ that is not a DVR, certainly the isomorphism classes containing $R$  and $\mathfrak{m}$ contain trace ideals: $R$ and $\mathfrak{m}$. However, not all isomorphism classes of ideals do. 

\begin{example}
For a field $k$, consider the ring 
\[R =  k [x,y,z]/(y^2-xz, x^2y-z^2, x^3-yz)   \cong  k[t^3, t^4, t^5].\]

The ideal $ (x,y)$ is not isomorphic to a trace ideal. For, if the isomorphism class of  an ideal $I$ contains a trace ideal, then $I \cong \T R I$. However \[\T R {(x,y)} = (x,y,z).\] 

To see this, recall that the trace ideal can be calculated from a presentation  matrix; see Remark \ref{tracepresmatrix}. One has:
 $$
 \xymatrix{R^3 \ar@[->][rrrr]^{A = \left[\begin{array}{ccc} -z & -y & x^2 \\ y & x& -z \end{array}\right]}&&&& R^2 \ar@{->>}[dr]^{\hspace{.3in}\circlearrowleft} \ar@{-->}[rr]^{\bar{\alpha}} && R .\\
&&&&& (x,y) \ar@{-->}[ur]_{\alpha}& \\
 }
 $$

Maps $\alpha \in \hom R {(x,y)} R$ induce and are induced by matrices $\bar{\alpha} \in \hom R {R^2} R$   in the left kernel of $A$. These are generated by $\begin{bmatrix}
     y& z  
  \end{bmatrix}$,
$\begin{bmatrix}
    x& y 
  \end{bmatrix}$, and 
  $\begin{bmatrix}
    z & x^2
  \end{bmatrix}$.

 \end{example}

\begin{remark}
We have seen that under the hypotheses of Theorem \ref{HW},  $\Z {\End R M}$ Gorenstein implies $\Z {\End R M}=R$, and this implies $M$ has a free summand. To prove Conjecture \ref{HWconj}, it is left to show that for reflexive modules over a one-dimensional Gorenstein ring, rigidity implies $\Z {\End R M}$ is Gorenstein, or more directly,  that over a one-dimensional Gorenstein domain, rigid implies balanced.

This investigation naturally leads to the the following question:
\end{remark}
\begin{Question}
Suppose $R$ is a Noetherian ring. What are the necessary and sufficient conditions for a rigid module to be balanced?
\end{Question}

\begin{example}
Let $R = k[[x,y]]/(xy)$ and $M= R/(x)$. Recall a projective resolution of $M$ over $R$ is \[ \cdots \overset{x}{\lra} R \overset{y}{\lra} R \overset{x}{\lra} R  \lra0 .\]
Applying $\hom R ? {R/(x)}$  yields the complex
\[0 \lra \hom R  R {R/(x)}  \overset{x}{\lra} \hom R  R {R/(x)}  \overset{y}{\lra} \hom R  R {R/(x)}   \lra \cdots.\]

Since multiplication by $y$ is an injective map on $R/(x)$ one gets \[\Ext 1 R {R/(x)}{ R/(x)} = 0 .\]
 
So $M$ is rigid, but not balanced since $\hom R {R/(x)} {R/(x)} \cong R/(x)$.\end{example}

\begin{definition} Set $E:= \End R M$. The module $\End E  M$ is called the \emph{double centralizer of $M$}.
\end{definition}

\begin{remark} 
There is a natural map  \[R \lra \End E  M\] given by sending $r\in R$ to multiplication by $r$.   Writing  $I$ for $\Ann M$,  this map induces a monomorphism of rings from $R/{I}$  into $E = \End R M  = \End {R/I} M$. It follows that the double centralizer, $\End E M \subseteq \End {R/I} M$, is precisely the $R/I$-endomorphisms of $M$ that commute with all the others. That is \[\Z{\End R M}  = \Z{\End {R/I} M}= \End E M.\]
 
 A module $M$ is said to have the \emph{Double Centralizer Property} (DCP) when the natural map $R\lra \End E M$ is a surjection. If $M$ is faithful, the map is always an injection and having the DCP is equivalent to being balanced.

\end{remark}

\begin{lemma} \cite[Lemma 2]{Suzuki1971} \label{Suzuki2}
Let $R$ be a ring (not necessarily commutative) and $M_R$ a right $R$-module. Assume that there exists the $R$-exact sequence $$ 0 \lra R_R \overset{\delta}{\lra} M_R$$ where $M_R$ is generated by $\delta(1)$ as an $E$-module. Then the following are equivalent

\begin{enumerate}
\item $M_R$ has the Double Centralizer Property
\item $M_R/\delta(R) \inj \displaystyle\prod M_R$ \qed
\end{enumerate}
\end{lemma}

\begin{remark} Suppose $R$ is a commutative ring.  If $\{x_1, \ldots, x_n\}$ generate $M$ as an $R$-module, there is an injection $$0 \lra R \overset{\delta}{\lra} M^{n}$$ where $\delta(1) = (x_1, \ldots, x_n)$.  This map fulfills the hypotheses of the lemma; see  \cite[Theorems 2.7, 2.8]{Koenig2001}. 

Recall, $\T R {M^n} = \T R M$. So when $M$ is a reflexive faithful $R$-module, \[\Z{\End  R M} = \Z{\End R {M^n}},\] and we may replace $M$ by $M^n$. 
Given $R$ injects into $M$ and Lemma \ref{Suzuki2},  proving Conjecture \ref{HWconj} is equivalent to showing $M/\delta(R)$ is torsionfree.  
 \end{remark}
\section*{Acknowledgements}

The results proved here form part of my doctoral dissertation  at the University
of Utah and is the culmination of work started while at the University of Nebraska-Lincoln. Special thanks to my Ph.D. advisor, Srikanth Iyengar. Many thanks to Susan Loepp for her helpful comments on several drafts of this paper. Thanks also to Alexandra Seceleanu for communicating Example \ref{grade}, to Steffen K{\"o}nig for suggesting several useful references and to Eleonore Faber and Roger Wiegand for their careful reading of this work. I am also very grateful to the referee for improving this paper through their detailed and constructive feedback.  

This work was partially supported by the NSF under grant DMS-1503044.   I also thank Williams College, where this paper was written, for its support through the Gaius Charles Bolin Fellowship. 
\bibliographystyle{amsplain}	

\begin{bibdiv}
\begin{biblist}

\bib{MaxOrders}{article}{
      author={Auslander, Maurice},
      author={Goldman, Oscar},
       title={Maximal orders},
        date={1960},
        ISSN={0002-9947},
     journal={Trans. Amer. Math. Soc.},
      volume={97},
       pages={1\ndash 24},
      review={\MR{0117252 (22 \#8034)}},
}

\bib{Bassubiquity}{article}{
      author={Bass, Hyman},
       title={On the ubiquity of {G}orenstein rings},
        date={1963},
        ISSN={0025-5874},
     journal={Math. Z.},
      volume={82},
       pages={8\ndash 28},
      review={\MR{0153708 (27 \#3669)}},
}

\bib{BandH}{book}{
      author={Bruns, Winfried},
      author={Herzog, J{\"u}rgen},
       title={Cohen-{M}acaulay rings},
      series={Cambridge Studies in Advanced Mathematics},
   publisher={Cambridge University Press},
     address={Cambridge},
        date={1993},
      volume={39},
        ISBN={0-521-41068-1},
      review={\MR{1251956 (95h:13020)}},
}

\bib{HIW}{unpublished}{
      author={Huneke, Craig},
      author={Iyengar, Srikanth},
      author={Wiegand, Roger},
       title={Rigid ideals over complete intersection rings},
        date={2016},
        note={In preparation},
}

\bib{HJsymmetry}{article}{
      author={Huneke, Craig},
      author={Jorgensen, David~A.},
       title={Symmetry in the vanishing of {E}xt over {G}orenstein rings},
        date={2003},
        ISSN={0025-5521},
     journal={Math. Scand.},
      volume={93},
      number={2},
       pages={161\ndash 184},
      review={\MR{2009580 (2004k:13039)}},
}

\bib{HWrigidityoftor}{article}{
      author={Huneke, Craig},
      author={Wiegand, Roger},
       title={Tensor products of modules and the rigidity of {${\rm Tor}$}},
        date={1994},
        ISSN={0025-5831},
     journal={Math. Ann.},
      volume={299},
      number={3},
       pages={449\ndash 476},
         url={http://dx.doi.org/10.1007/BF01459794},
      review={\MR{1282227 (95m:13008)}},
}

\bib{Kaplansky1954}{book}{
      author={Kaplansky, Irving},
       title={Infinite abelian groups},
   publisher={University of Michigan Press, Ann Arbor},
        date={1954},
      review={\MR{0065561 (16,444g)}},
}

\bib{Koenig2001}{article}{
      author={K{\"o}nig, Steffen},
      author={Slung{a}rd, Inger~Heidi},
      author={Xi, Changchang},
       title={Double centralizer properties, dominant dimension, and tilting
  modules},
        date={2001},
        ISSN={0021-8693},
     journal={J. Algebra},
      volume={240},
      number={1},
       pages={393\ndash 412},
         url={http://dx.doi.org/10.1006/jabr.2000.8726},
      review={\MR{1830559 (2002c:16018)}},
}

\bib{Lam}{book}{
      author={Lam, T.~Y.},
       title={Lectures on modules and rings},
      series={Graduate Texts in Mathematics},
   publisher={Springer-Verlag, New York},
        date={1999},
      volume={189},
        ISBN={0-387-98428-3},
         url={http://dx.doi.org/10.1007/978-1-4612-0525-8},
      review={\MR{1653294 (99i:16001)}},
}

\bib{LeuschkeCrepant}{incollection}{
      author={Leuschke, Graham~J.},
       title={Non-commutative crepant resolutions: scenes from categorical
  geometry},
        date={2012},
   booktitle={Progress in commutative algebra 1},
   publisher={de Gruyter, Berlin},
       pages={293\ndash 361},
      review={\MR{2932589}},
}

\bib{WLCMR}{book}{
      author={Leuschke, Graham~J.},
      author={Wiegand, Roger},
       title={Cohen-{M}acaulay representations},
      series={Mathematical Surveys and Monographs},
   publisher={American Mathematical Society, Providence, RI},
        date={2012},
      volume={181},
        ISBN={978-0-8218-7581-0},
         url={http://dx.doi.org/10.1090/surv/181},
      review={\MR{2919145}},
}

\bib{Mats}{book}{
      author={Matsumura, Hideyuki},
       title={Commutative ring theory},
     edition={Second},
      series={Cambridge Studies in Advanced Mathematics},
   publisher={Cambridge University Press},
     address={Cambridge},
        date={1989},
      volume={8},
        ISBN={0-521-36764-6},
        note={Translated from the Japanese by M. Reid},
      review={\MR{1011461 (90i:13001)}},
}

\bib{SallyVasc2}{article}{
      author={Sally, Judith~D.},
      author={Vasconcelos, Wolmer~V.},
       title={Stable rings and a problem of {B}ass},
        date={1973},
        ISSN={0002-9904},
     journal={Bull. Amer. Math. Soc.},
      volume={79},
       pages={574\ndash 576},
      review={\MR{0311643 (47 \#205)}},
}

\bib{Suzuki1971}{article}{
      author={Suzuki, Yasutaka},
       title={Dominant dimension of double centralizers},
        date={1971},
        ISSN={0025-5874},
     journal={Math. Z.},
      volume={122},
       pages={53\ndash 56},
      review={\MR{0289582 (44 \#6770)}},
}

\bib{Suzuki1974}{article}{
      author={Suzuki, Yasutaka},
       title={Double centralizers of torsionless modules},
        date={1974},
        ISSN={0021-4280},
     journal={Proc. Japan Acad.},
      volume={50},
       pages={612\ndash 615},
      review={\MR{0379602 (52 \#507)}},
}

\bib{Vasc1}{article}{
      author={Vasconcelos, Wolmer~V.},
       title={Reflexive modules over {G}orenstein rings},
        date={1968},
        ISSN={0002-9939},
     journal={Proc. Amer. Math. Soc.},
      volume={19},
       pages={1349\ndash 1355},
      review={\MR{0237480 (38 \#5762)}},
}

\bib{Vasccommend}{article}{
      author={Vasconcelos, Wolmer~V.},
       title={On commutative endomorphism rings},
        date={1970},
        ISSN={0030-8730},
     journal={Pacific J. Math.},
      volume={35},
       pages={795\ndash 798},
      review={\MR{0279086 (43 \#4812)}},
}

\bib{Vascaffine}{article}{
      author={Vasconcelos, Wolmer~V.},
       title={Computing the integral closure of an affine domain},
        date={1991},
        ISSN={0002-9939},
     journal={Proc. Amer. Math. Soc.},
      volume={113},
      number={3},
       pages={633\ndash 638},
         url={http://dx.doi.org/10.2307/2048595},
      review={\MR{1055780 (92b:13013)}},
}

\bib{Whitehead1980}{article}{
      author={Whitehead, James~M.},
       title={Projective modules and their trace ideals},
        date={1980},
        ISSN={0092-7872},
     journal={Comm. Algebra},
      volume={8},
      number={19},
       pages={1873\ndash 1901},
         url={http://dx.doi.org/10.1080/00927878008822551},
      review={\MR{588450}},
}

\bib{OrderIdeals}{article}{
    AUTHOR = {Evans, Jr., E. Graham and Griffith, Phillip A.},
     TITLE = {Order ideals},
 BOOKTITLE = {Commutative algebra ({B}erkeley, {CA}, 1987)},
    SERIES = {Math. Sci. Res. Inst. Publ.},
    VOLUME = {15},
     PAGES = {213--225},
 PUBLISHER = {Springer, New York},
      YEAR = {1989},
   MRCLASS = {13C10 (13D05 13D15)},
  MRNUMBER = {1015519},
MRREVIEWER = {Liam O'Carroll},
       DOI = {10.1007/978-1-4612-3660-3-10},
       URL = {http://dx.doi.org/10.1007/978-1-4612-3660-3-10},
}
\end{biblist}
\end{bibdiv}

\end{document}